\documentclass[11pt,a4paper]{elsarticle}
\usepackage{amsmath,amssymb}
\usepackage{amsthm}
\usepackage[a4paper]{geometry}
\usepackage{xcolor}
\usepackage{hyperref}
\allowdisplaybreaks
\newtheorem{theorem}{Theorem}
\newtheorem{proposition}{Proposition}[section]
\newtheorem{corollary}{Corollary}[section]
\newtheorem{lemma}{Lemma}[section]

\newtheorem{definition}{Definition}[section]
\usepackage{physics}
\theoremstyle{remark}
\usepackage{enumitem}
\newtheorem{remark}{Remark}[section]
\numberwithin{equation}{section}
\newcommand{\iu}{{i\mkern1mu}}

\begin{document}
\title{Limit behavior of a class of Cantor-integers}
\author[hzau]{Jin CHEN}
\ead{cj@mail.hzau.edu.cn}
\author[hzau]{Xin-Yu WANG}
\ead{xinyuwang@webmail.hzau.edu.cn}

\address[hzau]{College of Science, Huazhong Agricultural University, Wuhan 430070, China.}


\begin{abstract}
In this paper, we study a class of Cantor-integers $\{C_n\}_{n\geq 1}$ with the base conversion function $f:\{0,\dots,m\}\to \{0,\dots,p\}$ being strictly increasing and satisfying $f(0)=0$ and $f(m)=p$.  Firstly we provide an algorithm to compute the superior and inferior of the sequence $\left\{\frac{C_n}{n^{\alpha}}\right\}_{n\geq 1}$ where $\alpha =\log_{m+1}^{p+1}$, and obtain the exact values of the superior and inferior when $f$ is a class of quadratic function. Secondly we show that the sequence $\left\{\frac{C_n}{n^{\alpha}}\right\}_{n\geq 1}$ is dense in the close interval with the endpoints being its inferior and superior respectively. As a consequence, (i) we get the upper and lower pointwise density $1/\alpha$-density of the self-similar measure supported on $\mathfrak{C}$ at $0$, where $\mathfrak{C}$ is the Cantor set induced by Cantor-integers. (ii) the sequence $\{\frac{C_n}{n^{\alpha}}\}_{n\geq 1}$ does not have cumulative distribution function but have logarithmic distribution functions (given by a specific Lebesgue integral). Lastly we obtain the Mellin-Perron formula for the summation function of Cantor-integers. In addition, we investigate some analytic properties of the limit function induced by Cantor-integers.
\end{abstract}

\maketitle
\section{Introduction}

Given two positive integer $m$ and $p$ with $m<p$.  For any non negative integer $n$, the $(m+1)$-ary expansion of $n=\sum_{j=0}^{k}\epsilon_i (m+1)^j$ is denoted by
$[\epsilon_k \cdots \epsilon_0]_{m+1}$
for some $k\geq 0$, $\epsilon_j\in\{0,1,\dots,m\}$ and $\epsilon_k>0$. The integer sequence $\{C_n\}_{n\geq 0}$ is called \emph{Cantor-integers} if there exists a non-decreasing function $f:\{0,\dots,m\}\to \{0,\dots,p\}$ such that the $f(m)\leq p$ and $(p+1)$-ary expansion of $C_n$ is $[f(\epsilon_k) \cdots f(\epsilon_0)]_{p+1}$ whenever $n=[\epsilon_k \cdots \epsilon_0]_{m+1}$. Here $f$ will be called \emph{base conversion} function and note that $f$ is not unique.

 To avoid triviality we always assume $f(x)=x$ and $f(m)=p$ do not hold simultaneously. When $f(x)=2x$, $m=1$ and $p=f(m)=2$, we have that $C_{2n+i}=3C_{n}+2i$ for every $n\geq 0$ and $i\in \{0,1\}$. Namely, $\{C_n\}_{n\geq 0}$ is the only non-negative integers whose ternary expansion contains no $1$'s (cf. \cite[A005823]{SIS}), which is an analogue of Cantor’s triadic set. That's why it is called Cantor-integers.  Throughout the paper, set $\alpha:= \log_{m+1}^{p+1}$. It is easy to get that the order of growth for $C_n$ is $n^{\alpha}$ with $\alpha>1$. Recently, L\"{u} et al. in \cite{LCWW20} introduced quasi-linear sequence $g(n)$, in which the order of growth of quasi-linear sequence $\{g(n)\}$ must be strictly larger than $\{g(n)-g(n-1)\}$. They showed that 
\[ \left\{\frac{g(n)}{n^{\beta}}\right\}_{n\geq 1}\text{ is dense in }\left[\inf_{n\geq 1} \frac{g(n)}{n^{\beta}}, \sup_{n\geq 1}\frac{g(n)}{n^{\beta}}\right],\]
where $\beta$ is the order of growth of $g(n)$. It is not hard to check that  for every Cantor-integers sequence $\{C_n\}$, the order of growth of $\{C_n-C_{n-1}\}$ is also $\alpha$. This imlpies $C_n$ is not quasi-linear. In fact, as an arithmetic function,  $C_n$ has been widely investigated:
\begin{itemize}
\item When $f(x)=x$, $m=1$ and $p\geq f(m)+1=2$, $\{C_n\}_{n\geq 0}$ is the lexicographically earliest increasing sequence of nonnegative numbers that contains no arithmetic progression of length $3$ (cf. \cite[A005836]{SIS}). Flajolet et al. in \cite[Section 5]{FGK94} considered the Mellin-Perron formulae of the summation function $S(n):=\sum_{k<n} C_k$.
\item When $f(x)=2x$, $m=1$ and $p=f(m)+1=3$, then $\alpha =2$. Gawron and Ulas in \cite[Section 3]{GU16} investigated the density property of $\{C_n/n^2\}_{n\geq 1}$ (here they used the notation $\{b_k/k^2\}_{k\geq 1}$). They obtained that
\[ \left\{\frac{C_n}{n^2}\right\}_{n\geq 1}\text{ is contained and dense in }\left[\frac{2}{3},2\right].\]
\item When $f(x)=2x$ and $p\in\{f(m),f(m)+1\}$, Cao and Li in \cite{CL22}  generalized Gawron and Ulas's result in \cite[Section 3]{GU16} and reviewed the result about density property in the sense of self-similar measure. In details, they showed that
\[ \left\{\frac{C_n}{n^\alpha}\right\}_{n\geq 1}\text{ is contained and dense in }\left[\frac{f(m)}{p},f(1)\right], \]
and 
\[ \left\{\frac{x}{\mu_{\mathfrak{C}}^{\alpha}([0,x])}:x\in\mathfrak{C}\cap\left[\frac{2}{p+1},1\right] \right\} = \left[\frac{f(m)}{p},f(1)\right],\]
where $\mathfrak{C}$ and $\mu_{\mathfrak{C}}$ are defined as in \eqref{def:ifs} and \eqref{def:self-similar_measure} respectively.
\item When $f(x)=qx+r$ with $q\geq 2$ and $r\geq 0$, and $p+1\leq  q+f(m)$, Cao and Yu in \cite{CY22} generalized Cao and Li's results in \cite{CL22} and they showed that
\begin{equation}\label{eqn:linearcase_for_dense}
\left\{\frac{C_n}{n^\alpha}\right\}_{n\geq 1}\text{ is contained and dense in }\left[\frac{f(m)}{p},f(1)+\frac{f(0)}{p}\right],
\end{equation}

and 
\begin{equation}\label{eqn:linearcase_for_self-similar_measure}
\left\{\frac{x}{\mu_{\mathfrak{C}}^{\alpha}([0,x])}:x\in\mathfrak{C}\cap\left[\frac{f(1)}{p+1},1\right] \right\} = \left[\frac{f(m)}{p},f(1)+\frac{f(0)}{p}\right].
\end{equation}
\end{itemize}
Let the homogeneous Cantor set $\mathfrak{C} := \mathfrak{C}_{f,m,p}$ be the self-similar set generated by the iterated function system (IFS)
\begin{equation}\label{def:ifs}
 S_i(x) := \frac{x+f(i)}{1+p},i=0,\dots,m,
\end{equation}
and $\mu_{\mathfrak{C}}$ be the unique self-similar measure supported on $\mathfrak{C}$ satisfying 
\begin{equation}\label{def:self-similar_measure}
 \mu_{\mathfrak{C}} =\sum_{i=0}^m \frac{1}{m+1} \mu_{\mathfrak{C}}\circ {S_i}^{-1},
\end{equation}

Much research has focused on the Cantor set satisfying $S_0(0)=1$ and $S_m(1)=1$ (or equivalently $f(0)=0$ and $f(m)=p$) such as \cite{AS99} and \cite{KLY21}. In this paper, we are concerned with the case that $f$ is strictly increasing and satisfies $f(0)=0$ and $f(m)=p$.  The first term of $\{C_n\}_{n\geq 0}$ is always zero and will be omitted. For simplicity, write $\Sigma_m:=\{0,1,\dots,m\}$ and $\Sigma_m^{+} := \{1,\dots,m\}$. First we give an algorithm to compute the superior and inferior of $\{\frac{C_n}{n^{\alpha}}\}_{n\geq 1}$, and obtain the exact values when the base conversion function satisfies \eqref{Cond:1}.
\begin{theorem}\label{thm:boundofC_n/nalpha}

The superior and inferior of $\{\frac{C_n}{n^{\alpha}}\}_{n\geq 1}$ are 
\[ \max_{\epsilon\in\Sigma_m^{+}} \frac{C_{\epsilon}}{{\epsilon}^{\alpha}}
 \text{ and }\min_{1 \leq n < (m+1)^{\ell_0}} \frac{C_{n}+1}{(n+1)^{\alpha}}\] respectively,
 where $\ell_0$ is the smallest positive integer satisfying
 \[ \alpha\frac{(f(m)+1)^{\ell_0}}{(m+1)^{\ell_0+1}} \geq \max_{\epsilon \in \Sigma_m\backslash\{m\} }\frac{f(m)-f(\epsilon)}{m-\epsilon}. \]
In particular, when the base conversion function satisfies \eqref{Cond:1}, then we have
\[ \inf_{n\geq 1}\frac{C_n}{n^{\alpha}} =\min_{1\leq n < (m+1)^2} \frac{C_{n}+1}{(n+1)^{\alpha}} =\begin{cases}
1, &\text{ if }a\leq 0,\\
\min\left\{1,\min_{\epsilon\in\Sigma_m}\frac{C_{m+\epsilon+1}+1}{(m+\epsilon+2)^{\alpha}}\right\}, &\text{ if }a>0. \end{cases}\]
\end{theorem}
\begin{remark}
It should be noticed that $\ell_0$ as in Theorem \ref{thm:boundofC_n/nalpha} is not optimal. The experiment suggests that $\ell_0\leq 2$ for every strictly increasing $f$ satisfying $f(0)=0$ and $f(m)=p$. So far we can not prove this. However, it holds for quadratic $f$ satisfying \eqref{Cond:1}. For the exact values of the superior and inferior with respect to this kind of base conversion function, see Proposition \ref{prop:exact_sup_of_quadratic} and Proposition \ref{prop:exact_inf_of_quadratic} respectively. 
\end{remark}
Similarly with \cite[Theorem 2]{CL22} and \cite[Theorem 1.2]{CY22}, we also have the following density property.
\begin{theorem}\label{thm:denseproperty}
The sequence $\{\frac{C_n}{n^{\alpha}}\}_{n\geq 1}$ is dense in the close interval $[\inf_{n\geq 1}\frac{C_n}{n^{\alpha}}, \sup_{n\geq 1} \frac{C_n}{n^{\alpha}} ]$.
\end{theorem}

Recall that the Hausdorff dimension of $\mathfrak{C}$  is equal to $1/\alpha$, and the density of intervals of the form $[0,x]$ on the interval $(0,1]$ with respect to $\mu_{\mathfrak{C}}$  is defined by $d(x):=\frac{\mu_{\mathfrak{C}}([0,x])}{x^{1/\alpha}}$. Note that the lower and upper $1/\alpha$-density of $\mu_{\mathfrak{C}}$ at $0$ are defined by
\[ \Theta_{*}^{1/\alpha}(\mu_{\mathfrak{C}},0):=\liminf_{x\to0^+}d(x) \text{ and } \Theta^{*1/\alpha}(\mu_{\mathfrak{C}},0):=\limsup_{x\to 0^+}d(x) \]
respectively. For more details of pointwise density of self-similar measure, see \cite{FHW00} and \cite{KLY21}. Let \[\Lambda_0:=\{\ell:\text{ there exists a positive and vanishing sequence } \{x_n\}_{n\geq 1} \text{ such that } \lim_{n\to\infty}d(x_n) = \ell \}.\]
Now we have the following corollary about pointwise density of self-similar measure $\mu_{\mathfrak{C}}$ at $0$.
\begin{corollary}\label{cor:density_of_selfsimilarmeasure} We have that
\[ \Lambda_0 = \left[\inf_{n\geq 1}\frac{n}{{C_n}^{1/\alpha}},  \sup_{n\geq 1}\frac{n}{{C_n}^{1/\alpha}}\right]. \]
As a consequence,
\[ \Theta_{*}^{s}(\mu_{\mathfrak{C}},0) =\left(\sup_{n\geq 1} \frac{C_n}{n^{\alpha}}\right)^{-1/\alpha} \text{ and }\Theta^{*s}(\mu_{\mathfrak{C}},0)=\left(\inf_{n\geq 1} \frac{C_n}{n^{\alpha}}\right)^{-1/\alpha}.\]
\end{corollary}
\begin{remark}
Kong, Li and Yao in \cite{KLY21} considered the homogeneous Cantor set generated by the iterated function system
\[ \left\{ S_i(x)=\rho x+\frac{i(1-\rho)}{m}:i=0,\dots,m\right\},\]
where $m\geq 1$ and $\rho\in(0,1/(m+1)^2]$. If $\alpha \geq 2$, then we can take $\rho:=\frac{1}{f(m)+1}$. This implies that $f(i)=\frac{f(m)}{m}i$, which is the linear case when $m\mid f(m)$.
\end{remark}
Arithmetic functions related to number representation systems exhibit various periodicity phenomenon. For example, let $\{r(n)\}_{n\geq 0}$ be the Rudin-Shaipro sequence. Brillhart, Erd\'{o}s and Morton in \cite{BEM83} considered the properties of the partial sum $s(x):=\sum_{k=1}^{\lfloor x\rfloor}r(k)$ and $t(x):=\sum_{k=1}^{\lfloor x\rfloor}(-1)^kr(k)$, and introduced the limit functions
  \[ \tau(x):=\lim_{k\to\infty}\frac{s(\lfloor4^kx\rfloor)}{\sqrt{4^kx}} \text{ and } \mu(x):=\lim_{k\to\infty}\frac{t(\lfloor4^kx\rfloor)}{\sqrt{4^kx}},\]
which are defined for $x>0$. Note that $\tau(4x)=\tau(x)$ and  $\mu(4x)=\mu(x)$. They showed that $\tau(x)$ and $\mu(x)$ are continuous at every $x>0$, but are non-differentiable almost everywhere, and they also estimated the decay of coefficients of logarithmic Fourier series of $\tau(x)$ and $\mu(x)$. For the Cantor-integers sequence, write $C_{x} := C_{\lfloor x \rfloor}$ for every $x\in(0,\infty)$. Define a function $\lambda: (0,\infty) \to (0,\infty)$ defined by
\[  \lambda(x):= \lim_{k\to\infty} \frac{C_{(m+1)^kx}}{((m+1)^kx)^{\alpha}}.\]
It follows from \eqref{eqn:C_m+1((m+1)x)} that $\lambda(x)$ has a logarithmic Fourier series expansion. For the other analytic properties of $\lambda(x)$, see Proposition \ref{prop:cts}, Proposition \ref{prop:differentiability} and Proposition \ref{prop:fourierseries}.

Let the cumulative and logarithmic distribution function of the sequence $\{\frac{C_n}{n^{\alpha}}\}_{n\geq 1}$ be defined as in Section 7 and Section 8 in \cite{BEM83} respectively.  Using Theorem \ref{thm:denseproperty} and applying similar arguments of Theorem 1.4 and Theorem 1.5 in \cite{CY22}, we have the following two corollaries.
\begin{corollary}\label{cor:cumulative_dist}
The cumulative distribution function of the sequence $\{\frac{C_n}{n^{\alpha}}\}_{n\geq 1}$ does not exists for any point of $[\inf_{n\geq 1}\frac{C_n}{n^{\alpha}},\sup_{n\geq 1}\frac{C_n}{n^{\alpha}}]$.
\end{corollary}
\begin{corollary}\label{cor:logcumulative_dist}
If $\gamma \in [\inf_{n\geq 1}\frac{C_n}{n^{\alpha}},\sup_{n\geq 1}\frac{C_n}{n^{\alpha}}]$, then the logarithmic distribution function of the sequence $\{\frac{C_n}{n^{\alpha}}\}_{n\geq 1}$ exists at $\gamma$, and has the value 
\[ L(\gamma)=\frac{1}{\log(m+1)}\int_{E_{\gamma}}\frac{1}{x} \dd{x},\]
where $E_{\gamma}:=\{x\in[1,m+1]:\lambda(x)\leq \gamma\}$.
\end{corollary}

Similarly with \cite[Section 5]{FGK94}, we have the Mellin-Perron formula for the summation function of Cantor-integers. Let $\zeta(s,a)$ and $\zeta(s)$ be the well-known Hurwitz zeta function and Riemann zeta function  respectively. i.e.,
\[ \zeta(s,a):=\sum_{n=0}^{\infty}\frac{1}{(n+a)^s}\text{ and }\zeta(s):=\zeta(s,1).\] For simplicity, write $\Delta f(r):=f(r)-f(r-1)$ for $r\in\Sigma_m^+$. 
\begin{theorem}\label{thm:perron-mellin_formula}
For the summation function $S(n):=\sum_{1\leq k<n} C_n$, we have
\begin{equation*}
S(n)= n^{\alpha+1}F(\log_{m+1}^n)-n^2\frac{f(m)}{f(m)-m}-n\frac{\sum_{r=1}^m f(r)}{f(m)(m+1)}-\frac{f(m)(m+1)-\sum_{r=1}^m f(r)}{f(m)(m+1)+m},
\end{equation*}
where $F(u)$ is the Fourier series
\begin{equation*}
F(u) =\frac{1}{(f(m)+1)\log(m+1)}\sum_{k\in\mathbb{Z}} \frac{\sum_{r=1}^m \Delta f(r) \zeta(\gamma_k,\frac{r}{m+1})-f(m)\zeta(\gamma_k) }{\gamma_k(\gamma_k+1)}e^{2\pi i ku},
\end{equation*}
with $\gamma_k:=\alpha+\frac{2\pi k\iu}{\log(m+1)}$.
\end{theorem}

This paper is organized as follows. In Section 2, we prove Theorem \ref{thm:boundofC_n/nalpha}. In Section 3, we prove Theorem \ref{thm:denseproperty} and Corollary \ref{cor:density_of_selfsimilarmeasure}. In Section 4, we prove some some analytic properties of the limit function induced by Cantor-integers. In Section 5, we prove Theorem \ref{thm:perron-mellin_formula}.

\section{Proof of Theorem \ref{thm:boundofC_n/nalpha}}
With a little abuse of notation, write $m^\ell:=\underbrace{m\cdots m}_\ell$ for every $m\geq 1$. Given any finite word $u$, let $|u|$ denote the length of $u$.
We have the following observation for Cantor-integers. For simplicity, here we assume $n=[\epsilon_\ell\epsilon_{\ell-1}\cdots \epsilon_0]_{m+1}$ with $\epsilon_i \in \Sigma_m$ for some $0\leq i \leq \ell$.
For every integer $n\geq 1$ and $r\in\Sigma_m$, we have
\[ C_{{[\epsilon_\ell\cdots \epsilon_0 r]}_{m+1}}=(f(m)+1)C_{[\epsilon_\ell\cdots \epsilon_0 ]_{m+1}}+f(r).\] In  particular, for every integer $k\geq 1$, we have
\begin{equation}\label{Observ:1}
C_{{[\epsilon_\ell\cdots \epsilon_0 0^k]}_{m+1}}=(f(m)+1)^k C_{{[\epsilon_\ell\cdots \epsilon_0]}_{m+1}}.
\end{equation} 
First we give some useful propositions about Cantor-integers for general base conversion function.
\begin{proposition}\label{prop:C_n/nforngeq2}
For every $k\geq 2$ and $\epsilon_k\cdots\epsilon_1 \in \Sigma_m^{+}\times \Sigma_m^{k-1},$ we have
\begin{equation}
\frac{C_{[\epsilon_k\cdots\epsilon_1]_{m+1}}}{[\epsilon_k\cdots\epsilon_1]_{m+1}}\geq  \frac{f(m)+m+1}{2m+1}.
\end{equation}
\end{proposition}
\begin{proof}
Firstly we claim that for every $k\geq 2$ and $\epsilon_k\cdots\epsilon_1 \in \Sigma_m^{+}\times \Sigma_m^{k-1},$ we have
\begin{equation*}
\frac{C_{[\epsilon_k\cdots\epsilon_1]_{m+1}}}{[\epsilon_k\cdots\epsilon_1]_{m+1}} \geq \frac{C_{[\epsilon_k\epsilon_{k-1}]_{m+1}}}{[\epsilon_k\epsilon_{k-1}]_{m+1}}.
\end{equation*}
In fact, if $k\geq 3,$ we have 
\begin{eqnarray*}
\frac{C_{[\epsilon_k\cdots\epsilon_1]_{m+1}}}{[\epsilon_k\cdots\epsilon_1]_{m+1}} = \frac{(f(m)+1){C_{[\epsilon_k\cdots\epsilon_2]_{m+1}}}+f(\epsilon_1)}{(m+1)[\epsilon_k\cdots\epsilon_2]_{m+1}+\epsilon_1} 
\geq \frac{C_{[\epsilon_k\cdots\epsilon_2]_{m+1}}}{[\epsilon_k\cdots\epsilon_2]_{m+1}}.
\end{eqnarray*}
The last inequality is deduced using the fact that $f(m)>m$. Applying the above inequality $(k-2)$ times,  we have that the claim holds. Meanwhile, note that 
\begin{eqnarray*}
\frac{C_{[\epsilon_k\epsilon_{k-1}]_{m+1}}}{[\epsilon_k\epsilon_{k-1}]_{m+1}} &=&  \frac{(f(m)+1)f(\epsilon_k)+f(\epsilon_{k-1})}{(m+1)\epsilon_k+\epsilon_{k-1}}\\
&\geq & \frac{(f(m)+1)\epsilon_k+\epsilon_{k-1}}{(m+1)\epsilon_k+\epsilon_{k-1}}\\
&= & \frac{(f(m)+1)+\epsilon_{k-1}/\epsilon_k}{(m+1)+\epsilon_{k-1}/\epsilon_k}\\
& \geq & \frac{f(m)+1+m}{2m+1}.
\end{eqnarray*}
This completes the proof.
\end{proof}

\begin{proposition}For every $n\geq m+1$, we have
\begin{equation}\label{Ineqn:addm}
		 \frac{C_{(m+1)n+m}}{((m+1)n+m)^{\alpha}} \leq \frac{C_n}{n^{\alpha}}.
	\end{equation}
\end{proposition}
\begin{proof}
Note that 
\[ \frac{C_{(m+1)n+m}}{((m+1)n+m)^{\alpha}}=\frac{(f(m)+1)C_n+f(m)}{((m+1)n+m)^{\alpha}}=\frac{C_n+\frac{f(m)}{f(m)+1}}{(n+\frac{m}{m+1})^{\alpha}}.\]
By Bernoulli inequality, it suffices to show that 
\begin{equation*}
		 \frac{\alpha C_n}{n} \geq \frac{f(m)(m+1)}{(f(m)+1)m}.
\end{equation*}
Following from Proposition \ref{prop:C_n/nforngeq2}, we have
\[ \frac{C_n}{n}\geq \frac{f(m)+1+m}{2m+1}.\]
By the fact that $f(m)\geq m+1$, we have
\begin{equation} \label{ineqn:important1}
 \frac{f(m)+1+m}{2m+1} \geq \frac{f(m)(m+1)}{(f(m)+1)m}.
\end{equation}
This completes the proof.
\end{proof}
Similarly, we have the following results.
\begin{proposition}\label{prop:addepsilon}
There exists $k_0\in\mathbb{N}$ such that for every $n\geq (m+1)^{k_0}$ and every $\epsilon\in\{1,\cdots,m\}$, we have
\begin{equation*}
		 \frac{C_{(m+1)n+\epsilon}}{((m+1)n+\epsilon)^{\alpha}} \leq \frac{C_n}{n^{\alpha}}.
	\end{equation*}
\end{proposition}
\begin{proof}
As in the proof of (\ref{Ineqn:addm}), it suffices to show that there exists $n_0\in\mathbb{N}$ such that for every $n\geq n_0$ and every $\epsilon\in\{1,\cdots,m\}$,
\[  \frac{\alpha C_n}{n} \geq \frac{f(\epsilon)(m+1)}{\epsilon(f(m)+1)},\]
which follows from the fact that $\frac{C_n}{n} \to \infty$ when $n\to \infty$.
\end{proof}

\begin{proposition}\label{prop:monotoneforlargen}
There exists $k_1\in\mathbb{N}$ such that for every $n\geq (m+1)^{k_1}$, we have
\[ \frac{C_{(m+1)n}}{((m+1)n)^{\alpha}} > \frac{C_{(m+1)n+1}}{((m+1)n+1)^{\alpha}}> \cdots > \frac{C_{(m+1)n+m}}{((m+1)n+m)^{\alpha}}.\]
\end{proposition}
\begin{proof}

Given any $n=[\epsilon_{\ell}\cdots\epsilon_1]_{m+1}$ for some $\ell \geq 1$, we consider the function $T_{\ell}: [0,m] \to \mathbb{R}$ defined by 
\[ T_{\ell}(x):=\frac{C_{[\epsilon_{\ell}\cdots\epsilon_1x]_{m+1}}}{([\epsilon_{\ell}\cdots\epsilon_1x]_{m+1})^{\alpha}}=\frac{(f(m)+1)C_{[\epsilon_{\ell}\cdots\epsilon_1]_{m+1}}+f(x)}{((m+1)[\epsilon_{\ell}\cdots\epsilon_1]_{m+1}+x)^{\alpha}}. \]
Note that $\partial T_{\ell}(x)/\partial x < 0$ if and only if 
\[ f^{\prime}(x)\left((m+1)[\epsilon_{\ell}\cdots\epsilon_1]_{m+1}+x\right)-\alpha\left((f(m)+1)C_{[\epsilon_{\ell}\cdots\epsilon_1]_{m+1}}+f(x)\right)< 0. \]
In fact, following from that $C_{[\epsilon_{\ell}\cdots\epsilon_1]_{m+1}}/[\epsilon_{\ell}\cdots\epsilon_1]_{m+1} \to \infty (n\to \infty)$, if $\ell$ is sufficiently large, then for every $x\in[0,m]$, $\partial T_{\ell}(x)/\partial x < 0$. This completes the proof.
\end{proof}
 \subsection{The superior of \texorpdfstring{$\{\frac{C_n}{n^{\alpha}}:n\geq 1\}$}{}}
 For simplicity, set $M:=\max_{\epsilon\in\Sigma_m}\frac{C_\epsilon}{\epsilon^{\alpha}}$. We only need to show that for every $k\geq 0$, 
 \begin{equation*}
 \sup_{(m+1)^k\leq n < (m+1)^{k+1}} \frac{C_n}{n^{\alpha}} = M.
 \end{equation*}
 
 We will prove this by induction.  If $k=0$, it holds trivially. Assume the desired result holds for  $k\geq 0.$  By (\ref{Observ:1}), it suffices to consider $C_{(m+1)n+i}/((m+1)n+i)^{\alpha}$ for every $1\leq i\leq m$ and $(m+1)^k\leq n < (m+1)^{k+1}$.   In fact, note that
 \begin{eqnarray*}
 &&C_{(m+1)n+i}-M((m+1)n+i)^{\alpha} \\
 &=& (f(m)+1)C_n+f(i)-M((m+1)n+i)^{\alpha} \\
 						&\leq & (m+1)^{\alpha} M n^{\alpha}+Mi^{\alpha}-M((m+1)n+i)^{\alpha}\\
 						&=&M((m+1)n)^{\alpha}\left(1+\left(\frac{i}{(m+1)n}\right)^{\alpha}-\left(1+\frac{i}{(m+1)n}\right)^{\alpha}\right)\leq 0.
 \end{eqnarray*}
 The last inequality can be deduced from Bernoulli inequality. This completes the proof of the superior as in Theorem \ref{thm:boundofC_n/nalpha}. 
 
 It should be noticed that when $f$ satisfies \eqref{Cond:1}, we have the following result.
\begin{proposition}\label{prop:exact_sup_of_quadratic}
Define $T(x):=(2-\alpha)ax-b(\alpha-1)$. Then we have that
\begin{equation*}
\max_{\epsilon\in\Sigma_m^+}\frac{C_{\epsilon}}{\epsilon^{\alpha}} =\begin{cases}
f(1), &  \text{ if }a\leq 0 \text{ or }a=1,b=1,\\
\frac{f(m)}{m^{\alpha}}, &  \text{ if }a=2,b=-1,m=2 \text{ or }a=1,b=0,\\
f(1), &  \text{ if }a>0,\alpha\geq 2 \text{ and }b\geq 0,\\
\frac{f(m)}{m^{\alpha}}, &  \text{ if }a>0,\alpha\geq 2, b< 0 \text{ and } T(m)\geq 0,\\
f(1), &  \text{ if }a>0,\alpha\geq 2, b< 0 \text{ and } T(1)\leq 0,\\
\max\left\{\frac{f(\xi_m)}{{\xi_m}^{\alpha}},\frac{f(\xi_m+1)}{(\xi_m+1)^{\alpha}}\right\},  &  \text{ if }a>0,\alpha\geq  2, b< 0 \text{ and } T(m)<0<T(1).
\end{cases}
\end{equation*}
where $\xi_m:=\lfloor \frac{-b(\alpha-1)}{a(\alpha-2)}\rfloor$.
\end{proposition}
\begin{proof}
If $a>0,\alpha\geq 2$ and $b< 0$, the result follows from the fact that $T(m)\leq T(1)$.  The other cases can be verified directly.
\end{proof} 
  \subsection{The inferior of \texorpdfstring{$\{\frac{C_n}{n^{\alpha}}:n\geq 1\}$ for general $f$}{}}\label{sec:inf_for_general}
  It follows from (\ref{Ineqn:addm}) that for every finite word $u \in {\Sigma_m}^{|u|}$ with $|u|\geq 2$ and the first digit being positive,  we have
\[ \frac{C_{[u]_{m+1}}}{{[u]_{m+1}}^{\alpha}} \geq \frac{C_{[um]_{m+1}}}{{[um]_{m+1}}^{\alpha}} \geq \frac{C_{[umm]_{m+1}}}{{[umm]_{m+1}}^{\alpha}}\geq\cdots\geq \lim_{t\to\infty}\frac{C_{[um^t]_{m+1}}}{{[um^t]_{m+1}}^{\alpha}}=\frac{C_{[u]_{m+1}}+1}{([u]_{m+1}+1)^{\alpha}}.\]
 By Proposition \ref{prop:addepsilon} and Proposition \ref{prop:monotoneforlargen},  the inferior must be of the form $\frac{C_{[u]_{m+1}}+1}{([u]_{m+1}+1)^{\alpha}}$ for some $u$. 
It suffices to show that
if $|u|\geq \ell_0$ with $\ell_0$ defined as in Theorem \ref{thm:boundofC_n/nalpha}, then we have

\begin{equation}\label{ineqn:inf_monotone}
\min_{\epsilon \in \Sigma_m} \frac{C_{[u\epsilon]_{m+1}}+1}{([u\epsilon]_{m+1}+1)^{\alpha}} =\frac{C_{[um]_{m+1}}+1}{([um]_{m+1}+1)^{\alpha}} =\frac{C_{[u]_{m+1}}+1}{([u]_{m+1}+1)^{\alpha}}.
\end{equation}

 As in the proof of (\ref{Ineqn:addm}),  we only need to show that for every finite word $u$ with $|u|\geq \ell_0$ and every $\epsilon \in \Sigma_m\backslash \{m\}$,
\begin{equation*}
    \frac{ \alpha\left((f(m)+1)C_{[u]_{m+1}}+f(\epsilon)+1\right)}{(m+1)[u]_{m+1}+\epsilon+1}  \geq \frac{f(m)-f(\epsilon)}{m-\epsilon}.
\end{equation*}
In fact, if $|u|=k\geq \ell_0$, then we have
\begin{eqnarray*}
  \frac{ \alpha((f(m)+1)C_{[u]_{m+1}}+f(\epsilon)+1)}{(m+1)[u]_{m+1}+\epsilon+1}   &\geq& \frac{ \alpha(f(m)+1)C_{[u]_{m+1}}}{(m+1)([u]_{m+1}+1)}  \\
   &\geq& \frac{\alpha(f(m)+1)C_{[10^{k-1}]_{m+1}} }{(m+1)([m^k]_{m+1}+1)} \\
   &=& \frac{ \alpha(f(m)+1)^k}{(m+1)^{k+1}}\\
   &\geq&    \max_{\epsilon \in \Sigma_m\backslash\{m\} }\frac{f(m)-f(\epsilon)}{m-\epsilon} \geq  \frac{f(m)-f(\epsilon)}{m-\epsilon}.
\end{eqnarray*}
This completes the proof of the inferior for general $f$.
 \subsection{The inferior of \texorpdfstring{$\{\frac{C_n}{n^{\alpha}}:n\geq 1\}$}{} for  a class of quadratic \texorpdfstring{$f$}{}}
In this subsection, we will investigate the sequence $\{\frac{C_n}{n^{\alpha}}:n\geq 1\}$ with the  corresponding base conversion function $f$ being quadratic with zero constant term. i.e., $f(x)=ax^2+bx$ with $a,b\in \mathbb{Z}$ and
\begin{equation}\label{Cond:1}
f^{\prime}(x)=2ax+b>0\text{ for every }x \in [1,m].
\end{equation}
 If $a=0$ or $m=1$, then $f$ can be replaced with a linear function. Following from \eqref{eqn:linearcase_for_dense} or \cite[Theorem 1.1]{CY22}, we have $\inf_{n\geq 1}\frac{C_n}{n^{\alpha}}=1$.  Without loss of generality we always assume $a\neq 0$ and $m\geq 2$.
There are some useful facts for this kind of base conversion function:
\begin{itemize}
\item We have that $f(1)=a+b\geq 1$. 
\item When $1<\alpha<2$, we have that $ -am+2\leq b\leq (1-a)m+1$ since $m<f(m)< m^2+2m$. In details, if $a<0$, following from \eqref{Cond:1}, we have that $a=-1$ and $b\in \{2m,2m+1\}$; Similarly, if $a>0$, we have that either $a=2,b=-1,m=2$ or $a=1,b\in\{0,1\}$.
\item It is not hard to get that $f(m)\geq m^2$ and $\alpha \geq \frac{\log 5}{\log 3}$.
\end{itemize}

\begin{lemma}\label{lem:Cn/na_geq2}
For every $\epsilon_2\epsilon_1 \in \Sigma_m^{+}\times \Sigma_m$, we have
 \[ \inf_{\ell\geq 0}\left\{\frac{C_{[\epsilon_2\epsilon_1v]_{m+1}}}{{[\epsilon_2\epsilon_1v]_{m+1}}^{\alpha}}:v \in {\Sigma_m}^{\ell}  \right\} = \frac{C_{[\epsilon_2\epsilon_1]_{m+1}}+1}{([\epsilon_2\epsilon_1]_{m+1}+1)^{\alpha}}. \]
\end{lemma}
\begin{proof}
 As in the proof in Section \ref{sec:inf_for_general},  we only need to show that for every finite word $u$ with $|u|\geq 2$ and every $\epsilon \in \Sigma_m\backslash \{m\}$,
\begin{equation}\label{ineqn:inf_monotone_quadratic}
    \frac{ \alpha\left((f(m)+1)C_{[u]_{m+1}}+f(\epsilon)+1\right)}{(m+1)[u]_{m+1}+\epsilon+1}  \geq \frac{f(m)-f(\epsilon)}{m-\epsilon}=a(m+\epsilon)+b.
\end{equation}
Note that, following from the fact that $[u]_{m+1} \geq m+1$ and Proposition \ref{prop:C_n/nforngeq2}, we have
\begin{eqnarray*}
\frac{ \alpha\left((f(m)+1)C_{[u]_{m+1}}+f(\epsilon)+1\right)}{(m+1)[u]_{m+1}+\epsilon+1}
&\geq& \frac{\alpha(m+1)}{m+2}\frac{ (f(m)+1)C_{[u]_{m+1}}}{(m+1)[u]_{m+1}} \\
&\geq & \frac{\alpha(f(m)+1)(f(m)+m+1)}{(m+2)(2m+1)}.
 \end{eqnarray*}
Hence \eqref{ineqn:inf_monotone_quadratic} can be deduced from 
\begin{equation}\label{ineqn:inf_monotone_quadratic2}
 \frac{\alpha(f(m)+1)(f(m)+m+1)}{(m+2)(2m+1)}  \geq a(m+\epsilon)+b. 
\end{equation}
In fact, \eqref{ineqn:inf_monotone_quadratic} always holds for $\epsilon=0$, which follows from \eqref{ineqn:important1} and that fact that $\frac{\alpha(m+1)}{m+2}\geq 1$. It remains to consider the case $\epsilon \in \Sigma_m\backslash \{0,m\}$.
We distinguish two cases according to the sign of $a$.
\begin{description}
\item[Case 1: $a<0.$] Then  $a(m+\epsilon)+b\leq a(m+1)+b\leq \frac{f(m)}{m}$. Therefore by \eqref{ineqn:important1}, the desired inequality \eqref{ineqn:inf_monotone_quadratic2} holds.
\item[Case 2: $a>0.$] Then  $a(m+\epsilon)+b\leq a(2m-1)+b= \frac{2f(m)}{m}-f(1)$.
\item[Subcase 2(1): $\alpha\geq 2.$] Then $f(m)\geq m^2+2m$. Hence \eqref{ineqn:inf_monotone_quadratic2}  follows from 
\begin{eqnarray*}
\frac{2(f(m)+1)(f(m)+m+1)}{(m+2)(2m+1)}\geq \frac{2f(m)}{m} -1 \geq a(m+\epsilon)+b.
\end{eqnarray*}
\item[Subcase 2(2): $\alpha< 2.$] First we consider the case $a=2,b=-1,m=2$.
Then $f(m)=6$ and $\alpha=\frac{\log 7}{\log 3}> 1.7$. Note that we only need to verify \eqref{ineqn:inf_monotone_quadratic} for $\epsilon=1$. In other words,
\[ \frac{\alpha(7C_{[u]_3}+2)}{3[u]_3+2}\geq 5.\]
In fact, the above inequality follows from the fact that $C_{[u]_3}\geq \tfrac{9}{5}[u]_3$ by Proposition \ref{prop:C_n/nforngeq2}.
Now we investigate the case $a=1,b=1$. Then $f(m)=m^2+m$ and $\alpha=\frac{\log (m^2+m+1)}{\log (m+1)}> 1.7$.  Hence that 
\begin{eqnarray*}
 \frac{\alpha(f(m)+1)(f(m)+m+1)}{(m+2)(2m+1)} &=& \frac{\alpha(m^2+m+1)(m^2+2m+1)}{(m+2)(2m+1)}\\
&\geq & \frac{\alpha(m^2+m+1)m}{2m+1} \\
&\geq & 2m = \frac{2f(m)}{m}-f(1).
\end{eqnarray*}
At last, we focus on the case $a=1,b=0$. If $m=2$, then $\alpha =\frac{\log 5}{\log 3}$, and it suffices to check that 
\[ \frac{\alpha(5C_{[u]_3}+2)}{3[u]_3+2} \geq 3.\]
Following from Proposition \ref{prop:C_n/nforngeq2}, $C_{[u]_3}\geq \tfrac{7}{5}[u]_3$. This implies the  above inequality holds. Now we assume $m\geq 3$. Recall that $\frac{2f(m)}{m}-f(1)=2m -1$. Then we have
\begin{eqnarray*}
 \frac{\alpha(f(m)+1)(f(m)+m+1)}{(m+2)(2m+1)} &=& \frac{\alpha(m^2+1)(m^2+m+1)}{(m+2)(2m+1)}
\geq  2m -1,
\end{eqnarray*}
where the last inequality follows from the monotonicity of $\frac{\alpha(m^2+1)(m^2+m+1)}{(m+2)(2m+1)}-2m+1$ with respect to $m\geq 3$. This completes the proof.
\end{description}
\end{proof}
Set $\delta_a:=1$ if $a>0$ otherwise $\delta_a:=0$. We have the following lemma.
\begin{lemma}\label{lem:min_for_sencond_whenfirstdigit>1}
 If $\epsilon_2\in\{1+\delta_a,\cdots,m\}$, then we have
\begin{equation*}
\min_{\epsilon\in\Sigma_m}\frac{C_{[\epsilon_2\epsilon]_{m+1}}+1}{([\epsilon_2\epsilon]_{m+1}+1)^{\alpha}} =\frac{f(\epsilon_2)+1}{(\epsilon_2+1)^{\alpha}}. 
\end{equation*}
\end{lemma}
\begin{proof}
Fixed $\epsilon_2\in\{1+\delta_a,\cdots,m\}$. Like the proof of Lemma \ref{lem:Cn/na_geq2}, it suffices to show that for every $\epsilon \in \Sigma_m \backslash \{m\}$,
\begin{equation}\label{ineqn:inf_monotone_quadratic3}
    \frac{ \alpha\left((f(m)+1)f(\epsilon_2)+f(\epsilon)+1\right)}{(m+1)\epsilon_2+\epsilon+1}  \geq \frac{f(m)-f(\epsilon)}{m-\epsilon}=a(m+\epsilon)+b.
\end{equation}
\begin{description}
\item[Case 1: $a< 0.$ ] Then $a(m+\epsilon)+b\leq am+b=\frac{f(m)}{m}$. Note that $b\geq -2am$ if $a<-1$ otherwise $b\geq 2m$. Hence $f(\epsilon_2)/\epsilon_2\geq ma+b\geq m$ for every $\epsilon_2 \geq 1$. Hence that 
\begin{eqnarray*}
\frac{ \alpha\left((f(m)+1)f(\epsilon_2)+f(\epsilon)+1\right)}{(m+1)\epsilon_2+\epsilon+1} 
&\geq & \frac{ \alpha\left((f(m)+1)m\epsilon_2+\epsilon+1\right)}{(m+1)\epsilon_2+\epsilon+1} \\
&= & \frac{ \alpha\left((f(m)+1)m+(\epsilon+1)/\epsilon_2\right)}{m+1+(\epsilon+1)/\epsilon_2} \\
&\geq & \frac{ \alpha((f(m)+1)m+m) }{2m+1}\\
&\geq & \frac{f(m)}{m} \geq a(m+\epsilon)+b.
\end{eqnarray*}
\item[Case 2: $a>0.$]  Then $a(m+\epsilon)+b\leq a(2m-1)+b= \frac{2f(m)}{m}-f(1)$ and $f(\epsilon_2)/\epsilon_2\geq 2a+b$ for every $\epsilon_2\geq 2$. To prove \eqref{ineqn:inf_monotone_quadratic3}, we only need to show that
\begin{equation}\label{ineqn:inf_monotone_quadratic4}
 \frac{ \alpha((2a+b)(f(m)+1)+m/2)}{3m/2+1} \geq \frac{2f(m)}{m}-f(1). 
 \end{equation} 
This follows from that
\begin{eqnarray*}
\frac{ \alpha\left((f(m)+1)f(\epsilon_2)+f(\epsilon)+1\right)}{(m+1)\epsilon_2+\epsilon+1} 
&\geq & \frac{ \alpha\left((2a+b)(f(m)+1)\epsilon_2+\epsilon+1\right)}{(m+1)\epsilon_2+\epsilon+1} \\
&= & \frac{ \alpha\left((2a+b)(f(m)+1)+(\epsilon+1)/\epsilon_2\right)}{m+1+(\epsilon+1)/\epsilon_2}\\
&\geq&\frac{ \alpha((2a+b)(f(m)+1)+m/2)}{3m/2+1}.
\end{eqnarray*}
\item[Subcase 2(1): $\alpha\geq 2.$] Then \eqref{ineqn:inf_monotone_quadratic4} follows from the fact that $2a+b\geq 2$.
\item[Subcase 2(2): $\alpha< 2.$] First we consider the case $a=2,b=-1,m=2$ or $a=1,b=1$. Then $2a+b=3$ and $\alpha \geq \frac{\log 7}{\log 3}$. This implies \eqref{ineqn:inf_monotone_quadratic4} holds. Now for the case $a=1,b=0$. Then \eqref{ineqn:inf_monotone_quadratic3} is equivalent to 
\begin{eqnarray*}
\alpha((m^2+1)\epsilon_2^2+\epsilon^2+1) - (m+\epsilon)((m+1)\epsilon_2+\epsilon+1)\geq 0.
\end{eqnarray*}
Note that the left-hand side of the above inequality is decreasing with respect to $\epsilon\in[0,m-1]$, Hence that
\begin{eqnarray*}
&&\alpha((m^2+1)\epsilon_2^2+\epsilon^2+1) - (m+\epsilon)((m+1)\epsilon_2+\epsilon+1) \\ 
&\geq& \alpha((m^2+1)\epsilon_2^2+(m-1)^2+1) - (2m-1)((m+1)\epsilon_2+m)\\
&\geq& \alpha(4(m^2+1)+(m-1)^2+1) - (2m-1)(2(m+1)+m)\geq 0.
\end{eqnarray*}
This completes the proof.
\end{description}

\end{proof}
\begin{lemma}\label{lem:min_for_firstdigit}
We have
\begin{equation*}
\min_{x\in\{1,\cdots,m\}}\frac{f(x)+1}{(x+1)^{\alpha}} = 
\begin{cases}
1, & \text{ if }a<0,\\
\min\left\{1,\frac{a+b+1}{2^{\alpha}}\right\}, &  \text{ if }a>0.
\end{cases}
\end{equation*}
\end{lemma}
\begin{proof}
Consider the function $F(x):=\frac{f(x)+1}{(x+1)^{\alpha}}$ defined on $[1,m]$. It is easy to get that for each $x\in [1,m]$, $F^{\prime}(x)\leq 0$ if and only if $ f^{\prime}(x)(x+1)-\alpha(f(x)+1)\leq 0$.
Set $\tilde{F}(x):= f^{\prime}(x)(x+1)-\alpha(f(x)+1)$. Note that ${\tilde{F}}^{\prime}(x)=2a(x+1)-(\alpha-1)f^{\prime}(x)$.
\begin{description}
\item[Case 1: $a<0.$] Then ${\tilde{F}}^{\prime}(x)\leq 0$ for each $x\in [1,m]$. We will only need to show that ${\tilde{F}}(1)\leq 0.$ Since if this holds, then we have $\tilde{F}(x)\leq 0$ for each $x\in[1,m]$, which implies that
\[ \min_{x\in\{1,\cdots,m\}}\frac{f(x)+1}{(x+1)^{\alpha}} = 1. \]
In fact, we have
\begin{eqnarray*}
{\tilde{F}}(1) = 4a+2b-\alpha(a+b+1)
\leq
\begin{cases}
 2a-2\leq 0,  & \text{ if }\alpha\geq 2,\\
 \frac{4m^2}{(m^2+1)\log(m+1)}-4\leq 0,  & \text{ if }a=-1,b=2m,\\
 \frac{m(2m+1)}{(m^2+1)\log(m+1)}-4\leq 0,  & \text{ if }a=-1,b=2m+1.\\
\end{cases}
\end{eqnarray*} 
\item[Case 2: $a>0.$] 
\item[Subcase 2(1): $1<\alpha<2.$ ] In other words, $a=2,b=-1,m=2$ or $a=1,b\in\{0,1\}$. We have that ${\tilde{F}}^{\prime}(x)\geq 0$ for each $x\in [1,m]$. Note that ${\tilde{F}}(1)= 4a+2b-\alpha(a+b+1)\geq 0.$ This implies $\tilde{F}(x)\geq 0$ for each $x\in[1,m]$, Hence that
\[ \min_{x\in\{1,\cdots,m\}}\frac{f(x)+1}{(x+1)^{\alpha}} = \frac{a+b+1}{2^{\alpha}}. \]

\item[Subcase 2(2): $\alpha\geq 2.$] Note that
\[ \tilde{F}(x)= f^{\prime}(x)(x+1)-\alpha(f(x)+1)=(2-\alpha)ax^2+(2a+b-\alpha b)x+b-\alpha. \]
Hence if $\tilde{F}(1)\geq 0$, namely $\alpha\leq \frac{4a+2b}{a+b+1}$, then we have \[\min_{x\in[1,m]}F(x)=\min\{F(1),F(m)\}=\min\left\{1,\frac{a+b+1}{2^{\alpha}}\right\}.\]
Now it remains to consider the case $\alpha> \frac{4a+2b}{a+b+1}$ (or equivalently $\tilde{F}(1)< 0$).  Note that ${\tilde{F}}^{\prime}(x)\leq {\tilde{F}}^{\prime}(1)$ for each $x\in[1,m]$. We only need to show that ${\tilde{F}}^{\prime}(1)\leq 0$. Namely, $\alpha> \frac{6a+b}{2a+b}$. If this inequality holds, then we have $\tilde{F}(x)\leq 0$ for each $x\in[1,m]$. Hence that
\[ \min_{x\in\{1,\cdots,m\}}\frac{f(x)+1}{(x+1)^{\alpha}} = 1. \]
In fact, we have $\frac{4a+2b}{a+b+1}< \frac{6a+b}{2a+b}$ if and only if
$ 2a^2+ab-6a-b+b^2< 0 $
if and only if one of the following cases occurs: (i) $a=1,b\in\{0,1\}$; (ii) $a=2,b\in\{-1,0,1\}$; (iii) $a=3,b=-1$. For each case, it is not hard to get that $\alpha<\frac{4a+2b}{a+b+1}$ which contradicts the assumption. It follows that $\alpha > \frac{6a+b}{2a+b}$. This completes the proof.
\end{description}
\end{proof}
\begin{proof}[Proof of the inferior in Theorem \ref{thm:boundofC_n/nalpha} for $f$ satisfying \eqref{Cond:1}]
It follows from Lemma \ref{lem:Cn/na_geq2}, Lemma \ref{lem:min_for_sencond_whenfirstdigit>1} and Lemma \ref{lem:min_for_firstdigit} directly.
\end{proof}

When $f$ satisfies (\ref{Cond:1}), it should be noticed that the exact values of the inferior in Theorem \ref{thm:boundofC_n/nalpha} depend on the values of $a,b$ and $m$. In details, we will give the following result.
\begin{proposition}\label{prop:exact_inf_of_quadratic}
Given any $a>0$ and define \[T(x):=(2ax+b)(x+m+2)-\alpha\left((f(m)+1)f(1)+f(x)+1\right).\] Then we have that
\begin{equation*}
\min_{\epsilon\in\Sigma_m}\frac{C_{m+\epsilon+1}+1}{(m+\epsilon+2)^{\alpha}} =\begin{cases}
\frac{a+b+1}{2^{\alpha}}, &  \text{ if }T(m)\leq 0,\\
\min\left\{ \frac{C_{m+\xi_m+1}+1}{(m+\xi_m+2)^{\alpha}},\frac{C_{m+\xi_m+2}+1}{(m+\xi_m+3)^{\alpha}}\right\},  & \text{ otherwise}.
\end{cases}
\end{equation*}
where $\xi_m:=\min\{x\in\Sigma_m: T(x)\leq 0\text{ but } T(x+1)>0 \}$.
\end{proposition}
\begin{proof}
Consider the function $S(x):=\frac{C_{m+x+1}+1}{(m+x+2)^{\alpha}}$. Then we can get that $S^{\prime}(x)\leq 0$ if and only if $T(x)\leq 0$.  Rewriting the function $T$ by
\[ T(x):=(2-\alpha)ax^2+(2a(m+2)-b(\alpha-1))x+b(m+2)-\alpha((f(m)+1)f(1)+1). \] Clearly $T(0)\leq 0$. It suffices to investigate the case that $T(m)\leq 0$ and $\alpha>2$. Since if $T(m)> 0$ or $\alpha\leq 2$, then this proposition holds directly. Note that
\[ T^{\prime}(x):= 2ax(2-\alpha)+2a(m+2)-b(\alpha-1). \]
Without loss of generality we can assume that $T^{\prime}(0)\geq 0$ and $T^{\prime}(m)\leq 0$ (otherwise this proposition holds directly). This implies that $\alpha\geq 3$ if $b<0$. Now we claim that $T(x)\leq 0$ for every $x\in \mathbb{R}$. 
We only need to show that the discriminant $\Delta$ of $T$ is non-positive. In fact,
\begin{eqnarray*}
&& \frac{\Delta}{2a(\alpha-2)} \\
&:=& (2a(m+2)-b(\alpha-1))^2/(2a(\alpha-2))-2\alpha((f(m)+1)f(1)+1)+2b(m+2) \\
&\leq & m(2am+4a-b\alpha+b)-2\alpha(f(m)+2)+2b(m+2) ~\text{(By $T^{\prime}(0)\geq 0$ and $T^{\prime}(m)\leq 0$)}\\
&\leq & \begin{cases}
m(2am+4a-b)-4(am^2+bm+2)+2b(m+2)\leq 0, & \text{ if }b\geq 0,\\
2a(3-\alpha)+10(1-\alpha)\leq 0, &\text{ if }b< 0.
\end{cases}
\end{eqnarray*}
The last inequalities for the case $b<0$ can be deduced from the fact that $\alpha\geq 3$ and the monotonicity of $m(2am+4a-b\alpha+b)-2\alpha(f(m)+2)+2b(m+2)$ with respect to $b(\geq 1-a)$ and $m(\geq 2)$.
This completes the proof.
\end{proof}
\section{Proof of Theorem \ref{thm:denseproperty} and Corollary \ref{cor:density_of_selfsimilarmeasure}}
\begin{proof}[Proof of Theorem \ref{thm:denseproperty}]
For simplicity, set $\tilde{C}_n:=\frac{C_n}{n^{\alpha}}$ for every $n\geq 1$. By (\ref{Observ:1}) and the superior of $\{\tilde{C}_n\}_{n\geq 1}$, we have \[\limsup_{n\to \infty}  \tilde{C}_n = \sup_{n\geq 1}\tilde{C}_n.\]
Following from  (\ref{Ineqn:addm}), we have that \[\liminf_{n\to \infty}  \tilde{C}_n = \inf_{n\geq 1}\tilde{C}_n\] and
the limit $\lim_{k\to\infty}\tilde{C}_{[um^k]_{m+1}}$ exits for each $u \in\cup_{\ell\geq 0}\Sigma_m^{\ell}$.
 
Set $E:=\{\lim_{k\to\infty}\tilde{C}_{[um^k]_{m+1}}:u\in\cup_{\ell\geq 0}\Sigma_m^{\ell}\}$. Fixed any $\gamma \in (\inf_{n\geq 1}\tilde{C}_n, \sup_{n\geq 1} \tilde{C}_n)\backslash E$, we will construct a subsequence $\{\tilde{C}_{n_k}\}_{k\geq 1}$ of $\{\tilde{C}_n\}_{n\geq 1}$  converging to $\gamma.$ Let $k_0$ and $k_1$ be defined as Proposition \ref{prop:addepsilon} and Proposition \ref{prop:monotoneforlargen} respectively. Set $k_2:=\max\{k_0,k_1\}$ and $L(\gamma):=\{\tilde{C}_{\xi}:\xi\in\{1,\cdots,m\},\tilde{C}_{\xi} \geq \gamma  \}$. Now we take 
\[ n_1 := \min\left\{ \epsilon\in\{1,\cdots,m\}:\tilde{C}_{\epsilon}=\min L(\gamma)\right\},\]
and $n_k=(m+1)^{k-1}n_1$ for $2\leq k\leq k_2$. Assume that $n_k(k\geq k_2)$ has been chosen. We will choose $n_{k+1}$ recursively as follows:
\[ n_{k+1}=(m+1)n_k+\max\left\{ \epsilon\in\{0,1,\cdots,m\}:\tilde{C}_{(m+1)n_k+\epsilon}\geq \gamma \right\}.\]
Then we have that for every $k\geq 1,$ \[\tilde{C}_{n_{k+1}}\leq \tilde{C}_{n_k}.\] This implies that $\lim_{k\to \infty}\tilde{C}_{n_k}$ exists and is larger than $\gamma$. Now it suffices to show that
\[\lim_{k\to \infty}\tilde{C}_{n_k}\leq \gamma.\]
Define the set 
\[ \mathcal{S}:=\{k\in \mathbb{N}: n_k\neq m \mod (m+1)\}.\]
Note that $\mathcal{S}$ is non-empty. We claim that $\mathcal{S}$ is infinite. Otherwise if it is finite, then there exists $k^{*}$ such that for every $k\geq k^{*}$, $n_{k+1}=(m+1)n_k+m$. This contradicts with the value of $\gamma$. By Proposition \ref{prop:monotoneforlargen}, we have that for every $k\geq k_2,$
\[ \tilde{C}_{(m+1)n_k}> \tilde{C}_{(m+1)n_k+1}> \cdots > \tilde{C}_{(m+1)n_k+m}.\]
Hence that there exist infinite $k\geq k_2$ such that $\tilde{C}_{(m+1)n_k+m}<\gamma$. i.e.,
\[ \frac{(f(m)+1)C_{n_k}+f(m)}{((m+1)n_k+m)^{\alpha}}<\gamma.\]
Therefore, we have
\[ \lim_{k\to \infty} \tilde{C}_{n_k}=\lim_{k\to\infty} \frac{C_{n_k}}{{n_k}^{\alpha}}=\lim_{k\to\infty}\frac{(f(m)+1)C_{n_k}+f(m)}{((m+1)n_k+m)^{\alpha}} \leq \gamma.\]
This completes the proof.
\end{proof}
Now we can prove Corollary \ref{cor:density_of_selfsimilarmeasure}. 
\begin{proof}[Proof of Corollary \ref{cor:density_of_selfsimilarmeasure}] Recall that $d(x):=\frac{\mu_{\mathfrak{C}}([0,x])}{x^{1/\alpha}}$. Note that $S_0(0)=0$ and $S_m(1)=1$. By blow-up principle in \cite[Lemma 2.1]{AS99}, we have
\begin{eqnarray*}
\{d(x):x\in(0,1]\} =\left\{d(x):x\in\left[\frac{1}{f(m)+1},1\right]\right\}.
\end{eqnarray*}
Since $d(x)$ is continuous, the minimum and maximum of $d(x)$ on the interval  $[\frac{1}{f(m)+1},1]$ are attained, denoted by $d(x_m)$ and $d(x_M)$ respectively. Then we have that $x_m$ and $x_M$ are both contained in the set $\mathfrak{C}\cap [\frac{1}{f(m)+1},1] $. On the other hand, using the same arugment of \eqref{eqn:linearcase_for_self-similar_measure} or \cite[Corollary 1.3]{CY22}, we have that 
\[ \left\{\frac{x}{\mu_{\mathfrak{C}}^{\alpha}([0,x])}:x\in\mathfrak{C}\cap\left[\frac{f(1)}{f(m)+1},1\right] \right\} = \left[\inf_{n\geq 1}\frac{C_n}{{n}^{\alpha}},  \sup_{n\geq 1}\frac{C_n}{n^{\alpha}}\right]. \]
In other words, we have
\[ \left\{d(x):x\in\left[\frac{1}{f(m)+1},1\right] \cap \mathfrak{C} \right\}  =\left[\inf_{n\geq 1}\frac{n}{{C_n}^{1/\alpha}},  \sup_{n\geq 1}\frac{n}{{C_n}^{1/\alpha}}\right].\]
 This implies that 
\[ \{d(x):x\in(0,1]\} = \left\{d(x):x\in\left[\frac{1}{f(m)+1},1\right] \cap \mathfrak{C} \right\}. \]
Hence that
\[ \left[\liminf_{x\to0^+}d(x),\limsup_{x\to0^+}d(x)\right]\subset \{d(x):x\in(0,1]\}.\]
Following from the fact that $d\left(\frac{x}{(f(m)+1)^k}\right) = d(x)$ for any $k\geq 1$ and $x\in(0,1]$, we have
\[\{d(x):x\in(0,1]\} \subset  \left[\liminf_{x\to0^+}d(x) , \limsup_{x\to0^+}d(x)\right].\]
This completes the proof.
\end{proof}

\section{The limit function induced by Cantor-integers}
First we prove the existence of  $\lambda(x):= \lim_{k\to\infty} \frac{C_{(m+1)^kx}}{((m+1)^kx)^{\alpha}}$ for every $x>0$. In fact, we claim that
\begin{equation}\label{eqn:C(x)}
\lambda(x)=\frac{C_x}{x^{\alpha}}+\frac{\sum_{r=1}^{\infty}f(\epsilon_r)(f(m)+1)^{-r}}{x^{\alpha}},
\end{equation} 
where $x=[\epsilon_{-\ell}\cdots\epsilon_0.\epsilon_1\epsilon_2\cdots]_{m+1}:=\sum_{r=-\ell}^{\infty}\epsilon_r(m+1)^{-r}$ with $\epsilon_r \in \Sigma_{m}$ for each $ r\geq -\ell$.
In fact, $\lfloor (m+1)^k x \rfloor= [\epsilon_{-\ell}\cdots\epsilon_0\epsilon_1\cdots\epsilon_k]_{m+1}.$ Hence
\begin{eqnarray*}
\frac{C_{(m+1)^kx}}{((m+1)^kx)^{\alpha}} &=&\frac{(f(m)+1)^kC_{[\epsilon_{-\ell}\cdots\epsilon_0]_{m+1}}+\sum_{r=1}^{k}f(\epsilon_r)(f(m)+1)^{k-r} }{(f(m)+1)^k x^{\alpha}}\\
&=&\frac{C_{x}+\sum_{r=1}^{k}f(\epsilon_r)(f(m)+1)^{-r} }{ x^{\alpha}}.
\end{eqnarray*}
This imples (\ref{eqn:C(x)}) holds.  There are some observations for $\lambda(x)$.
\begin{itemize}
\item For every positive integer $n$, we have $\lambda(n)=\frac{C_n}{n^{\alpha}}$.
\item For every real number $x>0$, we have 
\begin{equation}\label{eqn:C_m+1((m+1)x)}
\lambda((m+1)x)=\lambda(x).
\end{equation}
\end{itemize}
 We know that the $(m+1)$-expansion is unique by excluding the expansions with a tail of $(m)$s. Recall that a real number $x$ is called $(m+1)$-rational if its $(m+1)$-expansion is finite, otherwise $x$ is called $(m+1)$-irrational. By (\ref{eqn:C_m+1((m+1)x)}), without loss of generality we can always assume that $x\in[1,m+1)$. For simplicity, define a function 
\begin{equation}\label{def:D(x)} D(x):=\sum_{r=1}^{\infty}f(\epsilon_r)(f(m)+1)^{-r} \text{ for }x= \displaystyle{\sum_{r=1}^{\infty}}\epsilon_r (m+1)^{-r} \text{ with }\epsilon_r\in \Sigma_m.
\end{equation}
Let $\mathfrak{C}:=\mathfrak{C}_{m,f}$ and $\mu_\mathfrak{C}$ be defined as in \eqref{def:ifs} and \eqref{def:self-similar_measure} respectively, define a function $g_{\mathfrak{C}}:[0,1]\to[0,1]$ by
\[ g_{\mathfrak{C}}(t):=\mu_{\mathfrak{C}}([0,t]), \]
which is a generalization of the Cantor function since $g_{\mathfrak{C}}(t)$ is the Cantor function if $\mathfrak{C}$ is the middle-thirds Cantor set. Equivalently,
\[ g_{\mathfrak{C}}(t)=\begin{cases}
\displaystyle{\sum_{r=1}^{\infty}}\epsilon_r (m+1)^{-r}, &t=\displaystyle{\sum_{r=1}^{\infty}}f(\epsilon_r)(f(m)+1)^{-r}\in \mathfrak{C} \text{ for } \epsilon_r \in\Sigma_m,\\
\displaystyle{\sup_{y\leq t,~y\in \mathfrak{C}}} g_{\mathfrak{C}}(y), &t\in[0,1]\backslash\mathfrak{C}.
\end{cases} \]
It is not hard to check that $D(x)=\sup\{t\in[0,1] : g_{\mathfrak{C}}(t)\leq x\}$. Hence the function D(x) is an inversion of the generalized Cantor function in some sense. Since $D(x)$ is strictly increasing, it is differentiable almost everywhere by the Lebesgue theorem about monotone function. This implies  $\lambda(x)$ is also differentiable almost everywhere. Moreover, for the continuity of $\lambda(x)$, we have the following explicit result by \cite[Remark 2]{CY22}.
\begin{proposition}\label{prop:cts}

\begin{enumerate}[label=(\roman*)]
\item The function $\lambda(x)$ is continuous at every positive $(m+1)$-irrational number.
\item The function $\lambda(x)$ is right continuous at every $(m+1)$-rational  number, and $\lambda(x)$ is left continuous at $x_0=[\epsilon_0.\epsilon_1\epsilon_2\cdots\epsilon_n]_{m+1}$ if and only if $\epsilon_n\in\{r\in\Sigma_m^+: \Delta f(r)=1\}$.
\end{enumerate}
\end{proposition}
At the same time, $\lambda(x)$ satisfies $(\alpha-\delta)$-H\"{o}lder condition almost everywhere for every sufficiently small $\delta>0$. To prove this, first we give the definition of normality.
\begin{definition}
Let $\ell \geq  1$, and let $B_{\ell}$ be a block of $\ell$ digits to the base $(m+1)(\text{recall that }m\geq 1)$. Also let $x_0 = \lfloor x_0 \rfloor+\sum_{r\geq 1}\epsilon_r(m+1)^{-r}$, and denote by $N(k, B_{\ell})$ the number of occurrences of the block $B_{\ell}$ in the initial block $\epsilon_1 \epsilon_2...\epsilon_k$ of $x_0 - \lfloor x_0 \rfloor$. Then $x_0$ is normal if and only if
\[ \lim_{k\to\infty} \frac{N(k,B_{\ell})}{k} = \frac{1}{(m+1)^{\ell}}, \]
for all $\ell\geq 1$ and all blocks $B_{\ell}$ of length $\ell$.
\end{definition}  

It is well-known  that almost all positive real numbers are normal. Now we have the following result.
\begin{proposition}\label{prop:differentiability}
If $x_0$ is normal to the base $(m+1)$, then we have that for every sufficiently small $\delta>0$,
\[ |\lambda(x_0+h)-\lambda(x_0)| = \mathcal{O}(|h|^{\alpha-\delta}), \text{ as }h\to 0, \]
where the implied constant depends only on $x_0$. 
\end{proposition}
\begin{proof}
It suffices to show that $|D(x_0+h)-D(x_0)| = \mathcal{O}(|h|^{\alpha-\delta})$  as $h\to 0$.
Let $x_0 = \lfloor x_0 \rfloor+\sum_{r\geq 1}\epsilon_r(m+1)^{-r}$, and assume $(m+1)^{-(n+1)}\leq h < (m+1)^{-n}$ for some $n\geq 1$. Then
\[ h=\sum_{r\geq n+1} h_r (m+1)^{-r}, 0\leq h_r\leq m, h_{n+1}\neq 0.\] 
Then 
\[ x_0+h=\lfloor x_0 \rfloor+\sum_{r= 1}^n\epsilon_r(m+1)^{-r}+ \sum_{r\geq n+1} (\epsilon_r+h_r) (m+1)^{-r}.\]
There maybe a carry into the $n$-th place in the right of the above equation. We need to estimate how long the carrying continues. Fixed any $\delta \in(0,1)$,  we may choose $n_0$ and  $B_{\ell}:= m^{\ell}$ (i.e., the $\ell$ consecutive block of $m$'s) such that  $\frac{2}{(m+1)^{\ell}}+\frac{\ell}{n_0}<\frac{\delta}{\alpha}$ and 
\[ N(k,B_{\ell})<\frac{2}{(m+1)^{\ell}}k, \text{ for }k\geq n_0. \]

Hence if $n\geq n_0$,  then there exists $t\in(1-\delta/\alpha,1)$ such that the number of occurrences of the block $B_{\ell}$ for $\epsilon_r$ between $t n$ and $n$ is at most \[\frac{2}{(m+1)^{\ell}}n<(1-t)n-\ell.\] Hence there is $r_0>tn$ for which $\epsilon_{r_0}\neq m$, and this implies that 
\[ x_0+h=\lfloor x_0 \rfloor+\sum_{r= 1}^{r_0-1}\epsilon_r(m+1)^{-r}+ \sum_{r\geq r_0} \epsilon_r^{*} (m+1)^{-r}\]
where $\epsilon_r^{*} \in \Sigma_m$ are digits. Therefore
\begin{eqnarray*}
|D(x_0+h)-D(x_0)|&=& \left|\sum_{r\geq r_0}(f(\epsilon)-f(\epsilon^*))(f(m)+1)^{-r}\right|\\
&\leq & (f(m)+1)^{-t n} = (m+1)^{-t\alpha n}\leq (1+f(m))h^{t\alpha}\leq (1+f(m))h^{\alpha-\delta},
\end{eqnarray*}
for $n\geq n_0.$ Thus
\[ |D(x_0+h)-D(x_0)| =\mathcal{O}(h^{\alpha-\delta}), \text{ as }h\to0^+.\]
A similar discussion shows that for every $h>0,$
\[ |D(x_0-h)-D(x_0)| =\mathcal{O}(h^{\alpha-\delta}), \text{ as }h\to0^+.\]
This completes the proof. 
\end{proof}

\begin{proposition}\label{prop:fourierseries}
The function $\lambda(x)$ has the logarithmic Fourier series expansion 
\begin{equation}
\lambda(x) \sim \sum_{n=-\infty}^{\infty} c_n e^{2\pi \iu \log x/\log(m+1)}
\end{equation}
 where 
\begin{equation}
c_n=\frac{1}{\log(m+1)}\int_1^{m+1}\frac{\lambda(x)}{x^{1-\alpha+\gamma_n}}\dd x \text{ with } \gamma_n:=\alpha+\frac{2\pi n\iu }{\log(m+1)},
\end{equation}
and where the infinite series converges to $\lambda(x)$ in the sense of Ces\'{a}ro sum at each continuity point as in Proposition \ref{prop:cts}. Moreover, $c_n=\mathcal{O}(1/|n|)$ as $n\to \infty$ or $n\to -\infty$.
\end{proposition}
\begin{proof}
Let $g(x):=\lambda((m+1)^{x/(2\pi)}).$ Then $g(x)$ has the period $2\pi$ and thus $g$  has a Fourier series
\[ g(x)\sim \sum_{n=-\infty}^{\infty}c_n e^{\iu n \theta},~\text{ with }c_n=\frac{1}{2\pi}\int_0^{2\pi} g(\theta)e^{-\iu n\theta} \dd \theta, \]
Then by \cite[Theorem 5.2 in p. 53]{SS03}, we have that the infinite series converges to $g(x)$ in the sense of Ces\'{a}ro sum at each point $x_0$ which makes $(m+1)^{x_0/(2\pi)}$ be a continuity point as in Proposition \ref{prop:cts}. Using
$\lambda(x) =g(2\pi \log x/\log(m+1))$, this easily yields the desired result. Recall that $D(x)$ is monotone, it follows from \cite{T67} that $c_n=\mathcal{O}(1/|n|)$.
\end{proof}

\section{Proof of Theorem \ref{thm:perron-mellin_formula}}
Following from \cite[Theorem 2.1]{FGK94} with the case $m=1$, we have the following lemma.
\begin{lemma} \label{lem:perron-mellin_Forumla} Let c>0 lie in the half-plane of absolute convergence of $\sum_k v_k k^{-s}$. Then we have
\begin{equation*}
\sum_{1\leq k<n} v_k\left(1-\frac{k}{n}\right) = \frac{1}{2\pi \iu}\int_{c-\iu\infty}^{c+\iu\infty} \left(\sum_{k\geq 1} \frac{v_k}{k^s}\right)n^s\frac{\dd s}{s(s+1)}
\end{equation*}
\end{lemma}
Consequently, we have the following identity (cf. \cite[Equation (2.6)]{FGK94}).
\begin{equation}\label{eqn:identity_for_perron_formula}
\int_{-1/4-\iu \infty}^{-1/4+\iu \infty} \zeta(s)n^s \frac{\dd s}{s(s+1)}=0.
\end{equation}

\begin{proof}[Proof of Theorme \ref{thm:perron-mellin_formula}]
Recall that $\Delta C_n:= C_n-C_{n-1}$ with $n\geq 1$. It is not hard to get that for every $n\geq 1$,
\[ \Delta C_{(m+1)n}=(f(m)+1)\Delta C_{n}-f(m), \]
and for every $n\geq 0$ and every $ r\in\{1,\dots, m\}$,
\[ \Delta C_{(m+1)n+r}=\Delta f(r).\]
It follows that 
\begin{eqnarray*}
\sum_{n=1}^{\infty} \frac{\Delta C_n}{n^s} &=&\sum_{r=1}^m\sum_{n=0}^{\infty} \frac{\Delta f(r)}{((m+1)n+r)^s}+\sum_{n=1}^{\infty}\frac{(f(m)+1)\Delta C_n-f(m)}{((m+1)n)^s} \\ 
&=&\sum_{r=1}^{m}\frac{\Delta f(r)}{(m+1)^s}\zeta\left(s,\frac{r}{m+1}\right)+\frac{f(m)+1}{(m+1)^s}\sum_{n=1}^{\infty} \frac{\Delta C_n}{n^s}-\frac{f(m)}{(m+1)^s}\zeta(s).
\end{eqnarray*}
Hence that
\[\sum_{n=1}^{\infty} \frac{\Delta C_n}{n^s} = \frac{\sum_{r=1}^m \Delta f(r) \zeta\left(s,\frac{r}{m+1}\right)-f(m)\zeta(s)}{(f(m)+1)((m+1)^{s-\alpha}-1)}.\]
By Lemma \ref{lem:perron-mellin_Forumla}, with $v_k=\Delta C_k$, for any $c> \alpha+1$,  we have
\begin{equation}\label{eqn:perron-mellin_for_cantor_integers}
\frac{S(n)}{n}= \frac{1}{2\pi \iu }\int_{c-\iu\infty}^{c+\iu\infty} \frac{\sum_{r=1}^m \Delta f(r) \zeta\left(s,\frac{r}{m+1}\right)-f(m)\zeta(s)}{(f(m)+1)((m+1)^{s-\alpha}-1)}n^s\frac{\dd s}{s(s+1)}
\end{equation}
The integrand in \eqref{eqn:perron-mellin_for_cantor_integers} has simple poles at $s=0$, $s=1$ and $s=\gamma_k$ for every $k\in \mathbb{Z}$. Shifting the line of integration to $\Re s=-\frac{1}{4}$ and taking residue into account, we get 
\begin{equation}\label{eqn:perron-mellin_for_cantor_integers2}
\frac{S(n)}{n}= \frac{\sum_{r=1}^m f(r)}{(m+1)f(m)}-2+F(\log_{m+1}^n)+R(n),
\end{equation}
where the Fourier series of $F(u)$ is 
\[ \frac{n^{\alpha}}{(f(m)+1)\log(m+1)}\sum_{k\in\mathbb{Z}} \frac{\sum_{r=1}^m \Delta f(r) \zeta\left(\gamma_k,\frac{r}{m+1}\right)-f(m)\zeta(\gamma_k)}{ \gamma_k(\gamma_k+1) }e^{2\pi \iu k u}, \]
which is the sum of residues of the integrand at the poles $s=\gamma_k$. We still have to analyse the remainder term 
\[ R(n):=\frac{1}{2\pi \iu (f(m)+1)}\int_{-1/4-\iu\infty}^{-1/4+\iu\infty} \frac{\sum_{r=1}^m \Delta f(r) \zeta\left(s,\frac{r}{m+1}\right)-f(m)\zeta(s)}{(m+1)^{s-\alpha}-1}n^s\frac{\dd s}{s(s+1)} .\] 
The integral converges since $|\zeta(-\frac{1}{4}+\iu t,a)-a^{\frac{1}{4}-\iu t}|\ll |t|^{3/4}$ where the constant does not depend on $a$ and $t$ (cf. \cite[Theorem 12.23]{A76}). Using the expansion 
\[ \frac{1}{(m+1)^{s-\alpha}-1}=-1-(m+1)^{s-\alpha}-(m+1)^{2s-2\alpha}-(m+1)^{3s-3\alpha}-\cdots, \]
in the integrand of $R(n)$, which is legitimate since $\Re (s-\alpha)<0$. By \eqref{eqn:identity_for_perron_formula}, we have
\begin{eqnarray*}
 R(n) &=& \frac{1}{2\pi \iu (f(m)+1)}\int_{-1/4-\iu\infty}^{-1/4+\iu\infty} \frac{\sum_{r=1}^m \Delta f(r) \zeta\left(s,\frac{r}{m+1}\right)}{(m+1)^{s-\alpha}-1}n^s\frac{\dd s}{s(s+1)} \\
 &=& - \sum_{k\geq 0} \frac{\sum_{r=1}^m \Delta f(r)R^{\prime}((m+1)^kn,r)}{(f(m)+1)^{k+1}},   \\
\end{eqnarray*}
where
\[ R^{\prime}(n,r):= \frac{1}{2\pi \iu }\int_{-1/4-\iu\infty}^{-1/4+\iu\infty}  \zeta\left(s,\frac{r}{m+1}\right) n^s\frac{\dd s}{s(s+1)} \]
Applying \cite[Equation (2.3)]{FGK94} by setting $\lambda_k\equiv 1,\mu_k\equiv k+\frac{r}{m+1}$, $x\equiv n^{-1}$ and $f(x) \equiv (1-x)H_0(x)$, we have 
\[ \sum_{1\leq k<n}\left(1-\frac{k+\frac{r}{m+1}}{n}\right)= \frac{1}{2\pi \iu } \int_{3/2-\iu\infty}^{3/2+\iu\infty}  \zeta\left(s,\frac{r}{m+1}\right) n^s\frac{\dd s}{s(s+1)}. \]
Shifting the the contour of $R^{\prime}(n,r)$ back to $\Re s=\frac{3}{2}$, by taking into account the residue at s$=0$ and $s=1$,
\[  R^{\prime}(n,r) = \frac{n+1}{2}-\frac{r}{m+1}+\sum_{1\leq k<n}\left(1-\frac{k+\frac{r}{m+1}}{n}\right)=n-\left(2-\frac{1}{n}\right)\frac{r}{m+1}.\]
Therefore, we have
\begin{eqnarray*}
 R(n) &=& - \sum_{k\geq 0} \frac{\sum_{r=1}^m \Delta f(r)\left((m+1)^k n-\left(2-\frac{1}{(m+1)^kn}\right)\frac{r}{m+1}\right)}{(f(m)+1)^{k+1}} \\
 &=& -n\frac{f(m)}{f(m)-m}+2-\frac{2\sum_{r=1}^m f(r)}{f(m)(m+1)}-\frac{1}{n}\frac{f(m)(m+1)-\sum_{r=1}^m f(r)}{f(m)(m+1)+m}
\end{eqnarray*}
Substituting $R(n)$ into \eqref{eqn:perron-mellin_for_cantor_integers2}, the proof is complete.
\end{proof}

\emph{Acknowledgments.} The first author thanks C.-Y. Cao for her valuable discussion. The work is supported by NSFC (Grant No. 11701202).

\end{document}